\documentclass[11pt]{amsart}

\usepackage{epigamath}

\usepackage[english]{babel}

\numberwithin{equation}{section}

\usepackage[shortlabels]{enumitem}
\setlist[enumerate,1]{label={\rm(\arabic*)}, ref={\rm\arabic*}}

\usepackage{amssymb}
\usepackage{pdflscape}
\usepackage{amscd}
\usepackage{graphicx}

\usepackage{tikz, tikz-3dplot, pgfplots}
\usepackage{tkz-graph}
\usepackage{tikz-cd}
\usetikzlibrary{positioning, calc, arrows, decorations.markings, decorations.pathreplacing}
\tikzset{
  symbol/.style={
    draw=none,
    every to/.append style={
      edge node={node [sloped, allow upside down, auto=false]{$#1$}}}
  }
}
\tikzcdset{scale cd/.style={every label/.append style={scale=#1},
    cells={nodes={scale=#1}}}}

\newtheorem{theorem}{Theorem}[section]
\newtheorem{corollary}[theorem]{Corollary}
\newtheorem{lemma}[theorem]{Lemma}

\DeclareMathOperator{\Hom}{Hom}
\DeclareMathOperator{\GL}{GL}
\DeclareMathOperator{\Gr}{Gr}
\DeclareMathOperator{\Ker}{Ker}
\DeclareMathOperator{\Coker}{Coker}

\renewcommand{\Im}{\operatorname{Im}}

\EpigaVolumeYear{9}{2025} \EpigaArticleNr{19} \ReceivedOn{November 14, 2024}
\InFinalFormOn{March 17, 2025}
\AcceptedOn{March 28, 2025}

\title{Motives of central slope Kronecker moduli}
\titlemark{Motives of central slope Kronecker moduli}

\author{Alexandre Astruc}
\address{Institut de Recherche Math\'ematique Avanc\'ee,
UMR 7501 Universit\'e de Strasbourg et CNRS, 
7 rue Ren\'e-Descartes, 67000 Strasbourg, France}
\email{a.astruc@unistra.fr}
 
\author{Fr\'ed\'eric Chapoton}
\address{Institut de Recherche Math\'ematique Avanc\'ee,
UMR 7501 Universit\'e de Strasbourg et CNRS,
7 rue Ren\'e-Descartes, 67000 Strasbourg, France}
\email{chapoton@unistra.fr}

\author{Karen Martinez}
\address{Ruhr University Bochum, Faculty of Mathematics, Universit\"atsstra{\ss}e 150, 44780 Bochum, Germany}
\email{Karen.Martinez@ruhr-uni-bochum.de}

\author{Markus Reineke}
\address{Ruhr University Bochum, Faculty of Mathematics, Universit\"atsstra{\ss}e 150, 44780 Bochum, Germany}
\email{markus.reineke@rub.de}

\authormark{A.~Astruc, F.~Chapoton, K.~Martinez, and M.~Reineke}

\AbstractInEnglish{We use dualities of quiver moduli induced by reflection functors to describe generating series of motives of Kronecker moduli spaces of central slope as solutions of algebraic and $q$-difference equations.}



\acknowledgement{F.\,C.~has benefited from the support of the French ANR project Charms (ANR-19-CE40-0017). The work of M.\,R.~is supported by the DFG project CRC-TRR 191 ``Symplectic structures in geometry, algebra and dynamics'' (281071066).}

\begin{document}



\maketitle

\begin{prelims}

\DisplayAbstractInEnglish

\bigskip

\DisplayKeyWords

\medskip

\DisplayMSCclass

\end{prelims}


\newpage

\setcounter{tocdepth}{1}

\tableofcontents


\section{Introduction}

Kronecker moduli are geometric invariant theory (GIT) quotients parametrizing equivalence classes of tuples of linear maps up to change of basis.  Despite being interesting in themselves as moduli spaces for a hard linear algebra problem, they have found applications in the theory of vector bundles on projective planes, see \cite{Drezet}, or more generally as ambient spaces for moduli spaces of semistable sheaves on projective varieties, see~\cite{ACK}, and to the tropical vertex and Gromov--Witten invariants of toric surfaces, see \cite{GP,ReinekePoisson}. Kronecker moduli can be studied with techniques of quiver moduli. For example, in the case where they are smooth and projective, their Poincar\'e polynomials can be computed; see \cite{ReinekeHNS}. In the special case of central slope (when the numerical parameters defining the space differ only by one), a simple formula for the Euler characteristic is proven in \cite{Weist} using torus localization techniques.

The present paper starts from the second-named author's observation that this Euler characteristic coincides with a number of intervals in higher Tamari lattices, see \cite{BMJ}, which makes it desirable to understand the Betti numbers of central slope Kronecker moduli better, with the ultimate aim of relating them directly to Tamari interval combinatorics. A crucial step in this direction is achieved in the present work, by describing the generating function of motives of central slope Kronecker moduli by an algebraic $q$-difference equation (in the spirit of a general result, see \cite{Ma}, about algebraicity of generating series of Euler characteristics of framed Kronecker moduli). The derivation of the $q$-difference equation is made possible by utilizing various dualities of Kronecker moduli, mainly originating in the reflection functors of quiver representation theory. More precisely, our main result reads as follows.

\begin{theorem}
Consider the moduli spaces $K_{d,d}^{(m),{\rm fr}}$ parametrizing stable $m$-tuples of linear maps on a $d$-dimensional complex vector space $W$, together with a framing vector in $W$, up to the action of\, $\GL(W)\times\GL(W)$ $($see Section~{\rm\ref{sec3}}$)$. Denote by
$$F(t)=1+\sum_{d\geq 1}\left[K_{d,d}^{(m),{\rm fr}}\right]_{\rm vir}t^d\in\mathbb{Q}(v)[[t]]$$the generating series of their virtual motives, where $v$ denotes a square root of the Lefschetz motive $($see Section~{\rm\ref{momo}}$)$. Then the series $F(t)$ is determined by $F(0)=1$ and
$$F(t)=\prod_{i=1}^m\left(1-v^{2i-m-1}t\prod_{j=1}^{m-2}F\left(v^{2i-2j-2}t\right)\right)^{-1}.$$
\end{theorem}

We recall all necessary notions of quiver representations, (framed) quiver moduli spaces, and their generating series of motivic invariants in Section~\ref{sec2}. Isomorphisms of quiver moduli spaces induced by reflection functors are constructed in Section~\ref{ref}. Specializing these to generalized Kronecker quivers, and thus to Kronecker moduli, in Section~\ref{sec3}, we obtain crucial dualities in Corollary~\ref{drk} and Theorem~\ref{reflectionframed}. These are used in Section~\ref{sec4} to establish a duality of generating series of motives, which is then specialized in Section~\ref{sec5} to the case of central slope, and related to Tamari intervals. Moreover, Section~\ref{sec5} discusses several other formulas for the motives and their generating series, as well as examples.

\subsection*{Acknowledgments}   The authors would like to thank the referees for suggesting several improvements of the exposition.

\section{Recollections on quiver moduli}\label{sec2}

\subsection{Basic quiver notation} For all basic notions of the representation theory of quivers, we refer to \cite{Schiffler}.  Let $Q$ be a finite acyclic quiver with set of vertices $Q_0$ and arrows written $\alpha\colon i\rightarrow j$.
Let $${\bf d}=\sum_{i\in Q_0}d_i{\bf i}\in\mathbb{N}Q_0$$ be a dimension vector for $Q$. We define the Euler form of $Q$ as the bilinear form on $\mathbb{Z}Q_0$ given by $$\langle{\bf d},{\bf e}\rangle=\sum_{i\in Q_0}d_ie_i-\sum_{\alpha\colon i\rightarrow j}d_ie_j,$$ and we denote its antisymmetrization by $$\{{\bf d},{\bf e}\}=\langle{\bf d},{\bf e}\rangle-\langle{\bf e},{\bf d}\rangle.$$ We fix complex vector spaces $V_i$ of dimension $d_i$ for all $i\in Q_0$. We define the variety of complex representations of $Q$ of dimension vector ${\bf d}$ as the complex affine space
$$R_{\bf d}(Q)=\bigoplus_{\alpha\colon i\rightarrow j}\Hom_\mathbb{C}(V_i,V_j),$$
on which the reductive complex algebraic group
$$G_{\bf d}=\prod_{i\in Q_0}\GL(V_i)$$
acts via change of basis
$$(g_i)_i\cdot(f_\alpha)_\alpha=(g_j\circ f_\alpha\circ g_i^{-1})_{\alpha\colon i\rightarrow j},$$
so that the $G_{\bf d}$-orbits in $R_{\bf d}(Q)$ correspond canonically to the isomorphism classes of complex representations of $Q$ of dimension vector ${\bf d}$. We can thus view a point $(f_\alpha)_\alpha\in R_{\bf d}(Q)$ as a representation $V=((V_i)_i,(f_\alpha)_\alpha)$ of $Q$ of dimension vector ${\bf d}$. 

\subsection{Moduli spaces of (semi-)stable representations}

We now summarize basic facts on moduli spaces; see \cite{King}. We fix a linear form $\Theta\in(\mathbb{Z}Q_0)^*$, called a stability. If $\Theta({\bf d})=0$, we define a point $(f_\alpha)_\alpha\in R_{\bf d}(Q)$, corresponding to a representation $V$,  to be $\Theta$-semistable if $$\Theta({\rm\bf dim}(U))\leq 0$$ for all subrepresentations $U\subset V$, that is, collections of subspaces $(U_i\subset V_i)_{i\in Q_0}$ such that $$f_\alpha(U_i)\subset U_j$$ for all arrows $\alpha\colon i\rightarrow j$ and  $${\rm\bf dim}(U)=\sum_i(\dim U_i){\bf i}.$$ Define $V$ to be $\Theta$-stable if $$\Theta({\rm\bf dim}(U))<0$$ for all non-zero proper subrepresentations.

More generally, fixing another linear form $\kappa\in(\mathbb{Z}Q_0)^*$ such that $\kappa({\bf d})>0$ for all $0\not={\bf d}\in\mathbb{N}Q_0$ (called a positive functional), we define the slope function $$\mu({\bf d})=\Theta({\bf d})/\kappa({\bf d})\in\mathbb{Q}$$ for $0\not={\bf d}\in\mathbb{N}Q_0$. Define $V$ as above to be $\mu$-semistable if $$\mu({\rm\bf dim}(U))\leq\mu({\rm\bf dim}(V))$$ for all non-zero subrepresentations $U\subset V$, and analogously for $\mu$-stability. If $\Theta({\rm\bf dim }(V))=0$, then $V$ is $\mu$-(semi-)stable if and only if it is $\Theta$-(semi-)stable. For later use, we note the following. 

\begin{lemma}\label{shiftstability}
The notion of $\mu$-$($semi-$)$stability does not change when $\Theta$ is replaced by
  $$\Theta'=a\Theta+b\kappa$$
  for $a,b\in\mathbb{Q}$ with $a>0$. In particular, the $\mu$-$($semi-$)$stable representations of a fixed dimension vector are precisely the $\Theta'$-$($semi-$)$stable ones for an appropriate $\Theta'$.
\end{lemma}

We denote by $R_{\bf d}^{\Theta\text{-}{\rm (s)st}}(Q)$ the (open, $G_{\bf d}$-stable) locus of $\Theta$-(semi-)stable points in $R_{\bf d}(Q)$, and similarly for~$R_{\bf d}^{\mu\text{-}{\rm (s)st}}(Q)$. We define the moduli space
$$M_{\bf d}^{\Theta\text{-}{\rm sst}}(Q)=R_{\bf d}^{\Theta\text{-}{\rm sst}}(Q)//G_{\bf d}$$
as the corresponding GIT quotient. The moduli space $M_{\bf d}^{\Theta\text{-}{\rm sst}}(Q)$, if non-empty, is an irreducible projective variety. It contains an open subset $M_{\bf d}^{\Theta\text{-}{\rm st}}(Q)$ which is the image of $R_{\bf d}^{\Theta\text{-}{\rm st}}(Q)$ under the quotient map
$$\pi\colon R_{\bf d}^{\Theta\text{-}{\rm sst}}(Q)\longrightarrow M_{\bf d}^{\Theta\text{-}{\rm sst}}(Q).$$
The restriction of $\pi$ to this open set is a principal bundle for the group
$$PG_{\bf d}=G_{\bf d}/\Delta,$$
where $\Delta$ is the diagonally embedded copy of $\mathbb{C}^*$ in $G_{\bf d}$. If non-empty, $M_{\bf d}^{\Theta\text{-}{\rm st}}(Q)$ is thus smooth and irreducible of dimension $1-\langle{\bf d},{\bf d}\rangle$. If the dimension vector ${\bf d}$ is $\Theta$-coprime, that is, $\Theta({\bf e})\not=0$ for all proper ${\bf e}\leq{\bf d}$ (componentwise inequality), the stable and semistable loci coincide, so $M_{\bf d}^{\Theta\text{-}{\rm sst}}(Q)=M_{\bf d}^{\Theta\text{-}{\rm st}}(Q)$ is smooth and projective. In this case, ${\bf d}$ is indivisible; that is, it is not a proper multiple of another dimension vector, and thus there exist universal bundles $\mathcal{V}_i$ of rank $d_i$ on $M_{\bf d}^{\Theta\text{-}{\rm sst}}(Q)$.

Changing the orientation of all arrows in the quiver $Q$ yields the opposite quiver $Q^{\rm op}$. Dualizing all vector spaces and linear maps in a quiver representation associates to $V$ a representation $V^*$ of $Q^{\rm op}$. On the geometric level, fixing ($\mathbb{C}$-bilinear) non-degenerate symmetric bilinear forms on the vector spaces $V_i$, we have a map $R_{\bf d}(Q)\rightarrow R_{\bf d}(Q^{\rm op})$ associating to $(f_\alpha)_\alpha$ the tuple of adjoint maps $(f_\alpha^*\colon V_j\rightarrow V_i)_{\alpha\colon i\rightarrow j})$, which is compatible with the $G_{\bf d}$-actions in the sense that, written symbolically, $$(g\cdot f)^*=\varphi(g)\cdot f^*$$ for the automorphism $\varphi$ of $G_{\bf d}$ which is the adjoint inverse in every component. We have thus proved the following. 

\begin{lemma}\label{duality}
There exists a natural isomorphism $$M_{\bf d}^{\Theta\text{-}{\rm sst}}(Q)\simeq M_{\bf d}^{(-\Theta)\text{-}{\rm sst}}(Q^{\rm op}).$$
\end{lemma}

\subsection{Motives of moduli spaces}\label{momo}
We consider the Grothendieck ring of varieties $K_0({\rm Var}_\mathbb{C})$ and denote by $[X]$ the class of a variety; in particular, $\mathbb{L}=[\mathbb{A}^1]$ denotes the Lefschetz motive, that is, the class of the affine line. We will work in the localization
$$R=K_0({\rm Var}_\mathbb{C})\left[\mathbb{L}^{\pm 1/2},\left(1-\mathbb{L}^i\right)^{-1},\, i\geq 1\right]$$ (which is essentially the Grothendieck ring of stacks with affine stabilizers, see \cite{Bridge}). For an irreducible variety $X$, we define its virtual motive
$$
[X]_{\rm vir}=\left(-\mathbb{L}^{1/2}\right)^{-\dim X}\cdot[X]\in R
$$
(in the present context, working with virtual motives has the advantage of making several formulas more symmetric).  It will turn out that all our calculations will already take place in the subring of $R$ generated by the rational functions in $\mathbb{L}^{1/2}$. Define the so-called motivic quantum affine space of $Q$ (see, for example, \cite{Moz}) as the formal power series ring $R[[x_i\, |\, i\in Q_0]]$, with multiplication twisted by the anti-symmetrized Euler form of $Q$:
$$x^{\bf d}\cdot x^{\bf e}=\left(-\mathbb{L}^{1/2}\right)^{\{{\bf d},{\bf e}\}}x^{{\bf d}+{\bf e}}.$$
We define the motivic generating series
$$A(x)=\sum_{{\bf d}\in\mathbb{N}Q_0}\frac
{\left[R_{\bf d}(Q)\right]_{\rm vir}}
{\left[G_{\bf d}\right]_{\rm vir}}
x^{\bf d}=\sum_{{\bf d}\in\mathbb{N}Q_0}\frac{\left(-\mathbb{L}^{1/2}\right)^{-\langle{\bf d},{\bf d}\rangle}x^{\bf d}}{\prod_{i\in Q_0}\left(\left(1-\mathbb{L}^{-1}\right)\cdot\ldots\cdot\left(1-\mathbb{L}^{-d_i}\right)\right)}$$ and, for all slopes $s\in\mathbb{Q}$,
$$A_s^\mu(x)=1+\sum_{{\bf d}\, :\,\mu({\bf d})=s}\frac{\left[R_{\bf d}^{\mu\text{-}{\rm sst}}(Q)\right]_{\rm vir}}{\left[G_{\bf d}\right]_{\rm vir}}x^{\bf d}.$$
We then have the following wall-crossing formula (which proves in particular that all motives of semistable loci are rational functions in $\mathbb{L}^{1/2}$). 

\begin{theorem}
In the motivic quantum affine space, we have
$$A(x)=\prod_{s\in\mathbb{Q}}^{\rightarrow}A_s^\mu(x),$$
where the product is taken in ascending order.
\end{theorem}

\begin{proof}
In \cite[Lemma 4.3]{ReinekePoisson} it is explained how such a factorization identity follows from a (Harder--Narasimhan-type) recursive formula given in \cite[Definition 4.1(2)]{ReinekePoisson}. That this recursion holds on the motivic level is explained in the proof of \cite[Theorem 3.5]{RSW}.
\end{proof}

\subsection{Framed moduli spaces}\label{subsec:framed}

The reference for the material of this section is \cite{EngelReineke}, with slight (but crucial to the following) adaptions of stability parameters which will be explained later. Given $(Q,{\bf d},\Theta)$ as before and another dimension vector $0\not={\bf n}\in\mathbb{N}Q_0$, we define a new datum $(\widehat{Q},\widehat{{\bf d}},\widehat{\Theta})$ as follows:
\begin{itemize}
\item The vertices of  $\widehat{Q}$ are those of $Q$, together with an additional vertex $0$.
\item The arrows of $\widehat{Q}$ are those of $Q$, together with $n_i$ arrows from $0$ to $i$, for all $i\in Q_0$.
\item We have $\widehat{d}_i=d_i$ for all $i\in Q_0$ and $\widehat{d}_0=1$.
\item We choose a positive functional $\kappa\in(\mathbb{Z}Q_0)^*$ and a positive integer $C$ such that $\kappa({\bf d})<C\cdot\gcd(\Theta)$ and define $\widehat{\Theta}_i=C\Theta_i-\kappa_i$ for all $i\in Q_0$ and $\widehat{\Theta}_0=\kappa({\bf d})$.
\end{itemize}

Representations of $\widehat{Q}$ of dimension vector $\widehat{\bf d}$ can be identified with pairs $(V,f)$ consisting of a representation $V$ of $Q$ of dimension vector ${\bf d}$, together with a tuple
$$f=\left(f_i\colon\mathbb{C}^{n_i}\rightarrow V_i\right).$$
With this identification, the following holds. 

\begin{lemma}\label{framed}
  The dimension vector $\widehat{\bf d}$ is $\widehat{\Theta}$-coprime, and the pair $(V,f)$ is $\widehat{\Theta}$-semistable if and only if\, $V$ is $\Theta$-semistable and  $\Theta({\rm\bf dim}(U))<0$ for all proper subrepresentations $U$ of\, $V$ containing the image of $f$, that is, satisfying $$\Im(f_i)\subset U_i\subset V_i$$ for all $i\in Q_0$.  
\end{lemma}

\begin{proof}
This follows from adapting the proof of \cite[Lemma 3.2]{EngelReineke}. There, a slope function $\widehat{\mu}=(\Theta+\epsilon{\bf 0}^*)/\dim$ is used. First, a direct inspection of this proof shows that the functional $\dim$ can be replaced by an arbitrary positive functional ${\kappa}$. Second, normalization of this stability to one evaluating to zero on $\widehat{\bf d}$, as provided by Lemma~\ref{shiftstability}, yields the choice of $\widehat{\Theta}$ above. \end{proof}

We define $$M_{{\bf d},{\bf n}}^{\Theta\text{-}{\rm fr}}(Q)=M_{\widehat{\bf d}}^{\widehat{\Theta}\text{-}{\rm sst}}(\widehat{Q})$$
and call it a framed moduli space. The forgetful map $(V,f)\mapsto V$ induces a projective morphism $$p\colon M_{{\bf d},{\bf n}}^{\Theta\text{-}{\rm fr}}(Q)\longrightarrow M_{\bf d}^{\Theta\text{-}{\rm sst}}(Q).$$
Note that the group $PG_{\widehat{\bf d}}\simeq G_{\bf d}$ is special (see \cite[Definition 3.7]{Bridge}), and $M_{{\bf d},{\bf n}}^{\Theta\text{-}{\rm fr}}(Q)$ is the geometric quotient of $R_{\widehat{\bf d}}^{\widehat{\Theta}\text{-}{\rm sst}}(\widehat{Q})$ by this group; thus
$$\left[M_{{\bf d},{\bf n}}^{\Theta\text{-}{\rm fr}}(Q)\right]_{\rm vir}=\left[R_{\widehat{\bf d}}^{\widehat{\Theta}\text{-}{\rm sst}}(\widehat{Q})\right]_{\rm vir}/[G_{\bf d}]_{\rm vir}.$$

We will use the following two results on framed moduli spaces. 

\begin{theorem}\label{theoremframed}
The following hold:
\begin{enumerate}
\item If\, ${\bf d}$ is $\Theta$-coprime, the map $p$ is a Zariski-locally trivial projective space fibration; more precisely, it is the total space of the projective bundle $$\mathbb{P}\left(\bigoplus_{i\in Q_0}\mathcal{V}_i^{n_i}\right).$$
\item\label{tf-2} For all $s\in\mathbb{Q}$, defining the generating series
  $$A_s^{\Theta\text{-}{\rm fr}}(x)=1+\sum_{\mu({\bf d})=s}\left[M_{{\bf d},{\bf n}}^{\Theta\text{-}{\rm fr}}(Q)\right]_{\rm vir}x^{\bf d},$$
  we have
$$A_s^{\Theta\text{-}{\rm fr}}(x)=A_s^\mu\left(\left(-\mathbb{L}^{1/2}\right)^{{\bf n}}x\right)\cdot A_s^\mu\left(\left(-\mathbb{L}^{1/2}\right)^{-{\bf n}}x\right)^{-1}$$
  in the motivic quantum affine space of\, $Q$.
\end{enumerate}
\end{theorem}

Here, for a series $F(x)=\sum_{\bf d}c_{\bf d}x^{\bf d}$, we define
$$F\left(\left(-\mathbb{L}^{1/2}\right)^{\bf n} x\right)=\sum_{\bf d}c_{\bf d}\left(-\mathbb{L}^{1/2}\right)^{{\bf n}\cdot{\bf d}}x^{\bf d}.$$

\begin{proof}
The first statement is \cite[Proposition 3.8]{EngelReineke}. The second statement follows from adapting \cite[Theorem 5.2]{EngelReineke} to the present definition of the motivic quantum affine space.
\end{proof}

\section{Reflection functors and moduli spaces}\label{ref}

We continue to denote by $Q$ a finite acyclic quiver. Let $i$ be a sink in $Q_0$. We define $s_iQ$ as the quiver with all arrows at $i$ reversed. Formally, we have $$(s_iQ)_0=Q_0,$$ and the arrows in $s_iQ$ are those of $Q$ not incident with $i$, together with arrows $\alpha^*\colon i\rightarrow j$ for every $\alpha\colon j\rightarrow i$ in $Q$. We define a reflection operator $s_i$ on $\mathbb{Z}Q_0$  by $(s_i{\bf d})_j=d_j$ for $j\not=i$ and $$(s_i{\bf d})_i=\left(\sum_{\alpha\colon j\rightarrow i}d_j\right)-d_i.$$

To define $s_i$ on $(\mathbb{Z}Q_0)^*$, we identify the latter with $\mathbb{Z}Q_0$ via the Euler form pairing. That is, if $\Theta=\langle\alpha,\_\rangle_Q$ for $\alpha\in\mathbb{Z}Q_0$, then
$$s_i\Theta=\langle s_i\alpha,\_\rangle_{s_iQ}.$$

Concretely, the following holds. 

\begin{lemma}\label{neu}
We have $(s_i\Theta)_i=-\Theta_i$ and, for all $j\not=i$,
$$(s_i\Theta)_j=\Theta_j+\sum_{\alpha\colon j\rightarrow i}\Theta_i.$$
\end{lemma}

Denote by $R_{\bf d}^{i,-}(Q)$ the open subset of $R_{\bf d}(Q)$ consisting of representations $V$ (again given by linear maps $f_\alpha\colon V_i\rightarrow V_j$ for $\alpha\colon i\rightarrow j$ in $Q$) such that the map $$\Phi_V=\bigoplus_{\alpha\colon j\rightarrow i}f_\alpha\colon \bigoplus_{\alpha\colon j\rightarrow i}V_j\longrightarrow V_i$$ is surjective.

For such a representation $V$, define a representation $S_i^+(V)$ of $s_iQ$ of dimension vector $s_i{\bf d}$ by $S_i^+(V)_j=V_j$ for $j\not=i$  and $$S_i^+(V)_i=\Ker(\Phi_V).$$ The maps $g_\alpha$ representing the arrows of $s_iQ$ in the representation $S_i^+(V)$ are given by $g_\alpha=f_\alpha$ if $\alpha$ is not incident with $i$, and $$g_{\alpha^*}\colon \Ker(\Phi_V)\longrightarrow V_j$$ is given by the projection to the component of $\bigoplus_{\alpha\colon j\rightarrow i}V_j$ corresponding to $\alpha$. Dually, for a source $i$ in $Q$, we define $s_iQ$, $s_i{\bf d}$, $s_i\Theta$, $R_{\bf d}^{i,+}(Q)$, and $S_i^-(V)$ accordingly.  The definitions of $S_i^+$ and $S_i^-$ are known to extend to functors inducing mutually inverse equivalences between appropriate subcategories of representations. We will now realize this correspondence of representations on the geometric level.

Denote by $Q'$ the full subquiver of $Q$ supported on $Q_0\setminus{\{i\}}$, and denote by ${\bf d}'$ the restriction of ${\bf d}$ to $Q'$. For a vector space $Y$ and $k\leq \dim Y$, denote by $\Gr_k(Y)$ (resp.~$\Gr^k(Y)$) the Grassmannian of $k$-dimensional subspaces (resp.~quotient spaces) of $Y$. We have the standard duality $$\Gr_k(Y)\simeq\Gr^{\dim Y-k}(Y)$$ which is $\GL(Y)$-equivariant.

We have a $\GL_{d_i}(\mathbb{C})$-principal bundle $$R_{\bf d}^{i,-}(Q)\longrightarrow R_{{\bf d}'}(Q')\times\Gr^{d_i}\left(\bigoplus_{\alpha\colon j\rightarrow i}V_j\right)$$
by mapping a representation $V$ to the restriction $V|_{Q'}$ together with $\Coker(\Phi_V)$,
and dually $$R_{s_i{\bf d}}^{i,+}(s_iQ)\longrightarrow R_{{\bf d}'}(Q')\times\Gr_{(s_i{\bf d})_i}\left(\bigoplus_{\alpha\colon j\rightarrow i}V_j\right).$$
Duality of Grassmannians identifies the targets of both maps $G_{{\bf d}'}$-equivariantly; we denote this common target by $\overline{R_{\bf d}(Q)}^i$. Composition with duality thus yields $G_{{\bf d}'}$-equivariant maps
$$R_{\bf d}^{i,-}(Q)\longrightarrow \overline{R_{\bf d}(Q)}^i\longleftarrow  R_{s_i\bf d}^{i,+}(s_iQ).$$

We now assume that $\Theta({\bf d})=0$ and that $\Theta_i<0$. In this case, it is immediately verified that $$R_{\bf d}^{\Theta\text{-}{\rm sst}}(Q)\subset R_{\bf d}^{i,-}(Q).$$ Dually, we find $$R_{s_i{\bf d}}^{s_i\Theta\text{-}{\rm sst}}(s_iQ)\subset R_{s_i{\bf d}}^{i,+}(s_iQ).$$ The images of both semistable loci under the above principal bundles coincide; we denote this image by $Z\subset\overline{R_{\bf d}(Q)}^i$ and find the following.

\begin{theorem}[\textit{cf.} \cite{Do}]\label{reflection}
We have an equality of localized motives
$$\left[R_{\bf d}^{\Theta\text{-}{\rm sst}}(Q)\right]/[G_{\bf d}]=\left[R_{s_i{\bf d}}^{s_i\Theta\text{-}{\rm sst}}(s_iQ)\right]/[G_{s_i{\bf d}}],$$
as well as an isomorphism of GIT quotients
$$M_{\bf d}^{\Theta\text{-}{\rm sst}}(Q)\simeq M_{s_i{\bf d}}^{s_i\Theta\text{-}{\rm sst}}(s_iQ).$$
\end{theorem}

\begin{proof}
Both sides of the equation of motives coincide with $[Z]/[G_{{\bf d}'}]$. In the same way, both sides of the isomorphism coincide with $Z//G_{{\bf d}'}$.
\end{proof}

\section{Dualities of Kronecker moduli}\label{sec3}

We now specialize the notions and results of the previous section to the $m$-arrow Kronecker quiver with two vertices $i$ and $j$, and $m$ arrows $\alpha_k\colon i\rightarrow j$ for $k=1,\ldots,m$. We fix a dimension vector ${\bf d}=d{\bf i}+e{\bf j}$ and vector spaces $V$, $W$ of dimensions $d$, $e$, respectively. We consider the stability $\Theta=e{\bf i}^*-d{\bf j}^*$. Then a representation, given by linear maps $f_1,\ldots,f_m\colon V\rightarrow W$, is $\Theta$-semistable if and only if $$\dim\sum_kf_k(U)\geq\frac{e}{d}\dim U$$ for all (non-zero, proper) subspaces $U\subset V$. We denote the semistable locus by $\Hom(V,W)^m_{\rm sst}$. For this datum, we consider the semistable moduli space $K_{d,e}^{(m)}$ defined as the quotient of $\Hom(V,W)^m_{\rm sst}$ by the structure group $\GL(V)\times\GL(W)$.

We consider the framing datum ${\bf n}={\bf j}$. By Lemma~\ref{framed}, the framed moduli space then parametrizes tuples $((f_k\colon V\rightarrow W)_k,w\in W)$ (up to the action of $\GL(V)\times \GL(W)$) such that $(f_k)_k$ defines a semistable representation, and if $w$ is contained in $\sum_kf_k(U)$, then this space has dimension strictly larger than $\frac{e}{d}\dim U$. We denote this framed moduli space by $K_{d,e}^{(m),{\rm fr}}$.

As a consequence of Lemma~\ref{duality} and Theorem~\ref{reflection} (by reflecting at the vertex $j$), we find the following. 

\begin{corollary}[\textit{cf.} \cite{Drezet1987}]\label{drk}
We have isomorphisms
$$K_{d,e}^{(m)}\simeq K_{e,d}^{(m)}\quad\text{and}\quad
K_{d,e}^{(m)}\simeq K_{md-e,d}^{(m)}$$
whenever $e\leq md$.
\end{corollary}

We now investigate in which case reflection induces an isomorphism even on the level of framed moduli spaces; this constitutes the technical heart of the paper.

\begin{theorem}\label{reflectionframed}
For $k\leq m$ and $d\geq 1$, we have an isomorphism of framed moduli spaces
$$K_{d,kd}^{(m),{\rm fr}}\simeq K_{d,(m-k)d+1}^{(m),{\rm fr}}.$$
\end{theorem}

\begin{proof}
We model the left-hand side as the following moduli space: we consider the extended quiver $\widehat{Q}$ given by
$$i\stackrel{(m)}{\longrightarrow}j\longleftarrow 0$$ with dimension vector $\widehat{\bf d}$ given by $$d{\bf i}+kd{\bf j}+{\bf 0}$$ and stability $\widehat{\Theta}$ given by $$(Ckd-\kappa_i){\bf i}^*+(-Cd-\kappa_j){\bf j}^*+d(\kappa_i+k\kappa_j){\bf 0}^*$$ for positive integers $C,\kappa_i,\kappa_j$ such that $\kappa_i+k\kappa_j<C$, compatible with the construction in Section~\ref{subsec:framed}. By Lemma~\ref{neu}, reflection at the sink ${j}$ yields the quiver $s_{ j}\widehat{Q}$ given by
$$i\stackrel{(m)}{\longleftarrow}j\longrightarrow 0,$$
the dimension vector $s_{j}\widehat{\bf d}$ given by $$d{\bf i}+((m-k)d+1){\bf j}+{\bf 0},$$ and the stability $s_{ j}\widehat{\Theta}$ given by $$(-C(m-k)d-\kappa_i-m\kappa_j){\bf i}^*+(Cd+\kappa_j){\bf j}^*+(-(C-\kappa_i)d+(kd-1)\kappa_j){\bf 0}^*.$$
Applying duality, we arrive at the quiver $(s_{ j}\widehat{Q})^{\rm op}=\widehat{Q}$ with dimension vector $s_{j}\widehat{{\bf d}}$ 
and stability $-s_{ j}\widehat{\Theta}$ given by
$$(C(m-k)d+\kappa_i+m\kappa_j){\bf i}^*-(Cd+\kappa_j){\bf j}^*+((C-\kappa_i)d-(kd-1)\kappa_j){\bf 0}^*.$$
On the other hand, again by Section~\ref{subsec:framed}, the right-hand side of the claimed isomorphism is modeled by the quiver $\widehat{Q}$, with dimension vector given by $s_{ j}\widehat{\bf d}$, and stability $\widehat{\Theta}'$ given by
$$(C'((m-k)d+1)-\kappa_i'){\bf i}^*+(-C'd-\kappa_j'){\bf j}^*+(\kappa_i'd+\kappa_j'((m-k)d+1)){\bf 0}^*$$
for positive integers $C',\kappa_i',\kappa_j'$ such that $\kappa_i'd+\kappa_j'((m-k)d+1)<C'$. We will now show that we can find parameters $C,\kappa_i,\kappa_j,C',\kappa_i',\kappa_j'$ such that $-s_{ j}\widehat{\Theta}=\widehat{\Theta}'$. Namely, we choose
$$\kappa_j=\kappa_i'=\kappa_j'=1,\quad \kappa_i=(m+1-k)d,\quad C=C'=(m+1-k)d+m+1,$$
and a direct verification shows that the necessary inequalities and the claimed equality are fulfilled.

Using the identification of framed Kronecker moduli as quiver moduli for $(s_j)\widehat{Q}$, as well as Lemma~\ref{duality} and Theorem~\ref{reflection}, we thus have a chain of isomorphisms
\begin{align*}
  K_{d,kd}^{(m),{\rm fr}}&\simeq M_{\widehat{\bf d}}^{\widehat{\Theta}\text{-}{\rm sst}}(\widehat{Q})\simeq M_{s_{\bf j}\widehat{\bf d}}^{s_{\bf j}\widehat{\Theta}\text{-}{\rm sst}}(s_{\bf j}\widehat{Q})\\
  &\simeq M_{s_{\bf j}\widehat{\bf d}}^{-s_{\bf j}\widehat{\Theta}\text{-}{\rm sst}}(\widehat{Q})\simeq M_{s_{\bf j}\widehat{\bf d}}^{\widehat{\Theta}'\text{-}{\rm sst}}(\widehat{Q})\simeq K_{d,(m-k)d+1}^{(m),{\rm fr}},
  \end{align*}
finishing the proof.
\end{proof}

 \section{Generating function identities}\label{sec4}
 
We introduce the following generating series in $R[[t]]$ for fixed $m$ and $1\leq k\leq m$ (the first two are specializations of the series $A_s^\mu(x)$ and $A_s^{\Theta\text{-}{\rm fr}}(x)$, respectively, introduced in Sections~\ref{momo} and~\ref{subsec:framed}):
\begin{eqnarray*}
A^{(k)}(t)&=&1+\sum_{d\geq 1}\frac{\left[\Hom\left({\bf C}^d,{\bf C}^{kd}\right)^m_{\rm sst}\right]_{\rm vir}}{\left[\GL_d(\mathbb{C})\times\GL_{kd}(\mathbb{C})\right]_{\rm vir}}t^d,\\
F^{(k)}(t)&=&1+\sum_{d\geq 1}\left[K_{d,kd}^{(m),{\rm fr}}\right]_{\rm vir}t^d,\\
G^{(k),\pm}(t)&=&1+\sum_{d\geq 1}\left[K_{d,kd\pm 1}^{(m)}\right]_{\rm vir}t^d.\end{eqnarray*}

To shorten notation, we will abbreviate $-\mathbb{L}^{1/2}$ to $v$.

We introduce an (ad hoc) operator $\nabla^{(k)}$ on formal series defined by
$$\nabla^{(k)}B(t)=\frac{vB(v^kt)-v^{-1}B(v^{-k}t)}{v-v^{-1}}.$$
It is linear and fulfills $\nabla^{(k)}t^d=[\mathbb{P}^{kd}]_{\rm vir}t^d$.
For a series $B(t)$ with constant term $B(0)=1$, we have
$$t\nabla^{(1)}\frac{B(t)-1}{t}=\Delta B(t)$$
for the operator $\Delta$ defined by
$$\Delta B(t)=\frac{B(vt)-B(v^{-1}t)}{v-v^{-1}}$$
(we note for later use that $\frac{1}{t}\Delta B(t)$ specializes to the standard derivative $B'(t)$ at $v=1$).

We can now specialize the generating series identities of the preceding section, and interpret the isomorphism of framed moduli spaces of Theorem~\ref{reflectionframed} as an identity of generating series. 

\begin{corollary}\label{corident}
We have the following identities:
\begin{eqnarray*}
A^{(k)}(t)&=&A^{(m-k)}(t),\\
F^{(k)}(t)&=&\frac{A^{(k)}(v^kt)}{A^{(k)}(v^{-k}t)},\\
G^{(k),-}(t)&=&G^{(m-k),+}(t),\\
F^{(k)}(t)&=&\nabla^{(m-k)}G^{(m-k),+}(t).
\end{eqnarray*}
\end{corollary}

\begin{proof} The first identity follows from the equality of motives in Theorem~\ref{reflection}. The second identity is a special case of the identity in Theorem~\ref{theoremframed}\eqref{tf-2}. The third identity follows again from Theorem~\ref{reflection}. The fourth identity follows from the first part of Theorem~\ref{theoremframed}.  
\end{proof}

We now combine all these identities. Because of the first and second identity, we find
\begin{align}\label{910}
  \prod_{i=1}^{m-k}F^{(k)}\left(v^{(m+1-k-2i)k}t\right)&=\frac{A^{(k)}\left(v^{(m-k)k}t\right)}{A^{(k)}\left(v^{-(m-k)k}t\right)}\\
 & =\frac{A^{(m-k)}\left(v^{(m-k)k}t\right)}{A^{(m-k)}\left(v^{-(m-k)k}t\right)}=\prod_{i=1}^kF^{(m-k)}\left(v^{(m-k)(k+1-2i)}t\right).\nonumber
\end{align}

Using the fourth identity, this equality reads

$$\prod_{i=1}^{m-k}\nabla^{(m-k)}G^{(m-k),+}\left(v^{(m+1-k-2i)k}t\right)=\prod_{i=1}^k\nabla^{(k)}G^{(k),+}\left(v^{(m-k)(k+1-2i)}t\right).$$

Using the third identity on the left-hand side, we arrive at the following identity between generating series attached to Kronecker moduli of slopes $k+1/d$ and $k-1/d$, respectively. 

\begin{corollary}\label{newduality}
We have
$$\prod_{i=1}^{m-k}\nabla^{(m-k)}G^{(k),-}\left(v^{(m+1-k-2i)k}t\right)=\prod_{i=1}^k\nabla^{(k)}G^{(k),+}\left(v^{(m-k)(k+1-2i)}t\right).$$
\end{corollary}

\section{Central slope}\label{sec5}

Now we specialize to the case $k=1$ in the identities of Corollary~\ref{corident}; note that this corresponds to the case of dimension vectors of slope~$1$ (which we call the central slope), thus $d=e$. We note the additional identity  
\begin{equation}\label{new}
  G^{(1),+}(t)=\frac{G^{(1),-}(t)-1}{t},
\end{equation}
which follows from Lemma~\ref{duality}. We then abbreviate $$F(t):=F^{(1)}(t)\quad\text{and}\quad G(t):=G^{(1),-}(t).$$ Then, by the third and fourth identities of Corollary~\ref{corident},

$$F(t)=\nabla^{(m-1)}G^{(m-1),+}(t)=\nabla^{(m-1)}G(t)$$
and, using again the fourth identity of Corollary~\ref{corident}, together with Equations (\ref{910}) and (\ref{new}), we obtain 
$$\prod_{i=1}^{m-1}F(v^{m-2i}t)=F^{(m-1)}(t)=\nabla^{(1)}G^{(1),+}(t)=\nabla^{(1)}\frac{G(t)-1}{t}=\frac{1}{t}\Delta G(t).$$

We thus arrive at our main result. 

\begin{theorem}\label{main}
The series $F(t)$ and $G(t)$, encoding the motives of framed Kronecker moduli of slope $1$ and the motives of Kronecker moduli of slope $1-1/d$, respectively, are mutually determined by
$$F(t)=\nabla^{(m-1)}G(t)\quad\text{and}\quad\Delta G(t)=t\prod_{i=1}^{m-1}F\left(v^{m-2i}t\right).$$
\end{theorem}

It is easily verified that the operators $\Delta$ and $\nabla^{(k)}$ commute, from which we immediately derive the following  $v$-difference equation determining $F(t)$. 

\begin{corollary}\label{vdifference}
The series $F(t)$ is determined by
$$\Delta F(t)=\nabla^{(m-1)}\left(t\prod_{i=1}^{m-1}F(v^{m-2i}t)\right)$$
 together with the initial condition $F(0)=1$.
 \end{corollary}
 
 Taking the definition of $F(t)$ as a generating series of virtual motives and comparing coefficients in the previous formula, we find the following recursion which allows for practical calculations. 
 
 \begin{corollary}\label{practical}
The virtual motives $m_d=[K_{d,d}^{(m),{\rm fr}}]_{\rm vir}$ are given recursively by $m_0=1$ and
 $$m_d=\frac{\left[\mathbb{P}^{(m-1)d}\right]_{\rm vir}}{\left[\mathbb{P}^{d-1}\right]_{\rm vir}}\sum_{d_1+\cdots+d_{m-1}=d-1}v^{\sum_i(m-2i)d_i}\prod_im_{d_i}.$$
 \end{corollary}

From Corollary~\ref{vdifference}, we will now derive an algebraic functional equation defining $F(t)$. 

\begin{theorem}
The series $F(t)$ is determined by $F(0)=1$ and
$$F(t)=\prod_{i=1}^m\left(1-v^{2i-m-1}t\prod_{j=1}^{m-2}F\left(v^{2i-2j-2}t\right)\right)^{-1}.$$
\end{theorem}

\begin{proof}
We abbreviate the right-hand side to $H(t)$; thus we have to prove that $H(t)=F(t)$. We obviously have $H(0)=1$, and we will now consider $\Delta H(t)$.

We abbreviate $$T_i:=1-v^{2i-m}t\prod_{j=1}^{m-2}F\left(v^{2i-2j-1}t\right)$$ for $i\geq 0$. Then $$H(vt)=T_1^{-1}\cdot\ldots\cdot T_m^{-1}\quad\text{and}\quad H(v^{-1}t)=T_0^{-1}\cdot\ldots \cdot T_{m-1}^{-1},$$
and thus
\begin{align*}
\Delta H(t)&=\frac{1}{v-v^{-1}}\left(T_1^{-1}\cdot\ldots\cdot T_m^{-1}-T_0^{-1}\cdot\ldots\cdot T_{m-1}^{-1}\right)\\
&=\frac{1}{v-v^{-1}}\cdot\frac{(T_0-T_m)\cdot T_1\cdot\ldots\cdot T_{m-1}}{T_0\cdot\ldots\cdot T_{m-1}\cdot T_1\cdot\ldots\cdot T_{m}}\\
&=\frac{H(vt)\cdot(T_0-T_m)}{(v-v^{-1})\cdot T_0}.
\end{align*}
Since
$$T_0=1-v^{-m}t\prod_{j=1}^{m-2}F\left(v^{-2j-1}t\right),\quad T_m=1-v^m\prod_{j=1}^{m-2}F\left(v^{2m-1-2j}t\right),$$
we have
\begin{equation*}
  \begin{aligned}
\Delta H(t)&=\frac{H(vt)}{(v-v^{-1})\cdot T_0}\left(v^mt\prod_{j=1}^{m-2}F\left(v^{2m-1-2j}t\right)-v^{-m}t\prod_{j=1}^{m-2}F\left(v^{-2j-1}t\right)\right)\\
&=\frac{H(vt)}{F(vt)\cdot(v-v^{-1})\cdot T_0}\left(v^mt\prod_{j=1}^{m-1}F\left(v^{2m-1-2j}t\right)-v^{-m}t\prod_{j=1}^{m-2}F\left(v^{-2j-1}t\right)\cdot F(vt)\right)\\
&=\frac{H(vt)}{F(vt)\cdot(v-v^{-1})\cdot T_0}
\begin{aligned}[t]
&\Bigg(\underbrace{v^mt\prod_{j=1}^{m-1}F\left(v^{2m-1-2j}t\right)-v^{-m}t\prod_{j=1}^{m-1}F\left(v^{-2j+1}t\right)}_{\mbox{(I)}}\\
&\hphantom{(}+\underbrace{v^{-m}t\prod_{j=1}^{m-1}F\left(v^{-2j+1}t\right)-v^{-m}t\prod_{j=1}^{m-2}F\left(v^{-2j-1}t\right)\cdot F(vt)}_{\mbox{(II)}}\Bigg).\end{aligned}
  \end{aligned}
  \end{equation*}
The right-hand side of the $v$-difference equation defining $F(t)$ in Corollary~\ref{vdifference}  reads
$$\nabla^{(m-1)}\left(t\prod_{j=1}^{m-1}F\left(v^{m-2j}t\right)\right)=\frac{v^mt\prod_{j=1}^{m-1}F\left(v^{2m-1-2j}t\right)-v^{-m}t\prod_{j=1}^{m-1}F\left(v^{1-2j}t\right)}{v-v^{-1}},$$
which we recognize as $\frac{1}{(v-v^{-1})}\mbox{(I)}$, whereas $\mbox{(II)}$ equals $$(T_0-1)\cdot\left(F(vt)-F\left(v^{-1}t\right)\right).$$ We thus find
$$\Delta H(t)=\frac{H(vt)}{F(vt)\cdot T_0}\left(\Delta F(t)+(T_0-1)\cdot\Delta F(t)\right)=\frac{H(vt)}{F(vt)   }\Delta F(t).$$
This easily implies that $$\Delta\left(\frac{H(t)}{F(t)}\right)=0;$$ thus $H(t)$ is a scalar multiple of $F(t)$, thus equal to $F(t)$ since $H(0)=1$.
\end{proof}

We illustrate these formulas with some examples for $m=3$ and small $d$. We only list the coefficients of the motives, viewed as Laurent polynomials in $v$:

\begin{description}[leftmargin=2\parindent,labelindent=2\parindent]
\item[$K_{1,1}^{(3),{\rm fr}}$] $1,1,1$
\item[$K_{2,2}^{(3),{\rm fr}}$] $1,2,3,3,3,2,1$
\item[$K_{3,3}^{(3),{\rm fr}}$] $1,2,5,8,11,12,13,12,11,8,5,2,1$,
\item[$K_{4,4}^{(3),{\rm fr}}$] $1,2,5,10,18,28,40,50,58,62,64,62,58,50,40,28,18,10,5,2,1$
\item[$K_{1,0}^{(3)}$] $1$
\item[$K_{2,1}^{(3)}$] $1,1,1$
\item[$K_{3,2}^{(3)}$] $1,1,3,3,3,1,1$
\item[$K_{4,3}^{(3)}$] $1,1,3,5,8,10,12,10,8,5,3,1,1$
\item[$K_{5,4}^{(3)}$] $1,1,3,5,10,14,23,30,41,46,51,46,41,30,23,14,10,5,3,1,1$
\end{description}

Finally, we specialize the above functional equations to $v=1$; geometrically, this corresponds to passing from virtual motives to Euler characteristics. Our aim is to give a new proof of \cite[Theorem 6.6]{Weist} avoiding the iterated torus fixed-point localization techniques used there.

We denote the specialized series by $\overline{F}(t)$ and $\overline{G}(t)$:
$$\overline{F}(t)=\left(1-t\overline{F}(t)^{m-2}\right)^{-m}\quad\text{and}\quad\overline{G}(t)'=\overline{F}(t)^{m-1}.$$
Since $\overline{F}(0)=1$, the series $\overline{F}(t)^{1/m}$ exists and fulfills the equation
$$\overline{F}(t)^{1/m}=1+t\left(\overline{F}(t)^{1/m}\right)^{(m-1)^2}.$$
To solve this functional equation, we substitute $t=x^{(m-1)^2}$ and apply the Lagrange inversion formula in the following form:

Suppose that series $U(x),V(x)\in\mathbb{Q}[[x]]$ with $V(0)\not=0$ are related by $$U(x)=V\left(xU(x)\right).$$ Then, for all $k,d\in\mathbb{Z}$, we have
$$(k+d)[x^d]U(x)^k=k[x^d]V(x)^{k+d},$$
where $[x^d]U(x)$ denotes the $x^d$-coefficient of the series $U(x)$.

Applying this to $$U(x)=\overline{F}\left(x^{(m-1)^2}\right)^{1/m},\quad V(x)=1+x^{(m-1)^2},\quad k=m(m-1)$$ and substituting back to the variable $t$, we find
$$
[t^d]\overline{F}(t)^{m-1}=\frac{m(m-1)}{m(m-1)+(m-1)^2d}\binom{(m-1)^2d+m(m-1)}{d}.
$$
Finally, using $d[t^d]\overline{G}(t)=[t^{d-1}]\overline{G}'(t)$, we find that the $t^d$-coefficient in $\overline{G}(t)$ equals
$$\frac{m}{d((m-1)d+1)}\binom{(m-1)^2d+m-1}{d-1}.$$
After some cancellations, we see that this equals
$$\frac{m-1}{d((m-2)d+1)}\binom{(m-1)^2d+m-2}{d-1},$$
which rederives \cite[Theorem 6.6]{Weist}.

Using the calculation of the number of intervals in generalized Tamari lattices in \cite{BMJ}, we have thus proved the following. 

\begin{corollary}
The Euler characteristic of\, $K_{d,d-1}^{(m)}$ equals the number of intervals in the $(m-2)$-Tamari lattice of index $d$.
\end{corollary}

For the case $m=3$, see sequences A000260, A255918 in \cite{oeis}. The generalized Tamari lattices are defined by a covering relation on generalized Dyck paths, and play a central role in a fascinating set of conjectures related to multivariate diagonal harmonics; see \cite{BP}. It is thus desirable to find a statistic on Tamari intervals whose partition function in $v$ equals the motive of central slope Kronecker moduli. These combinatorial ramifications of the present work will be pursued elsewhere.


\newcommand{\etalchar}[1]{$^{#1}$}

\end{document}